\documentclass[12pt]{article}
\usepackage[margin=1.5in]{geometry} 
\usepackage{amsmath,amsthm,amssymb,amsfonts}
\usepackage{changepage}
\usepackage{tikz}
\usepackage[shortlabels]{enumitem}
\usetikzlibrary{matrix} 
 
\newcommand{\T}{\mathcal{T}}

\newcommand{\Z}{\mathbb{Z}}

\newcommand{\id}{\textnormal{id}}

\newcommand{\inv}{^{-1}}

\newcommand{\cay}{\textnormal{Cay}}

\newcommand{\res}{\upharpoonright}

\newtheorem{theorem}{Theorem}

\newtheorem{corollary}{Corollary}
\newtheorem{lemma}{Lemma} 
\newtheorem{prop}{Proposition}  
\newtheorem{prob}{Problem} 
 
\begin{document}
 
\title{Descriptive Chromatic Numbers of Locally Finite and Everywhere Two Ended Graphs}
\author{Felix Weilacher}
\maketitle

\begin{abstract}
\noindent We construct Borel graphs which settle several questions in descriptive graph combinatorics. These include ``Can the Baire measurable chromatic number of a locally finite Borel graph exceed the usual chromatic number by more than one?'' and ``Can marked groups with isomorphic Cayley graphs have Borel chromatic numbers for their shift graphs which differ by more than one?'' We also provide a new bound for Borel chromatic numbers of graphs whose connected components all have two ends.
\end{abstract}


\section{Introduction}\label{sec:intro}


A \textit{graph} on a set $X$ is a symmetric irreflexive relation $G \subset X \times X$. In this situation, the elements of $X$ are called the \textit{vertices} of $G$. Vertices $x$ and $y$ are called \textit{adjacent} if $(x,y) \in G$, and in this case the pair $\{x,y\}$ is called an \textit{edge} of $G$. The \textit{degree} of a vertex is the number of other vertices adjacent to it. $G$ is called \textit{locally finite} if every vertex has finite degree, is said to have \textit{bounded degree} $d$ if every vertex has degree at most $d$, and is called \textit{$d$-regular} if every vertex has degree exactly $d$, where $d$ is some natural number. A \textit{connected component} of $G$ is an equivalence class of the equivalence relation generated by $G$. 

A (proper) \textit{coloring} of $G$ is a function, say, $c:X \rightarrow Y$ to some set $Y$ such that if $x$ and $y$ are adjacent, $c(x) \neq c(y)$. In this situation, the elements of $Y$ are called \textit{colors}. The sets $c\inv(\{y\})$ for $y \in Y$ are called \textit{color sets}. If $|Y| = k$, $c$ is called a $k$-coloring. The \textit{chromatic number} of $G$, denoted $\chi(G)$, is the least $k$ such that $G$ admits a $k$-coloring.

Descriptive graph combinatorics studies these notions in the descriptive setting: Let $X$ now be a Polish space. A graph $G$ on $X$ is called \textit{Borel} if $G$ is Borel in the product space $X \times X$. A coloring $c:X \rightarrow Y$ is called \textit{Borel} if $Y$ is also a Polish space and $c$ is a Borel function. The \textit{Borel chromatic number} of $G$, denoted $\chi_B(G)$, is the least $k$ such that $G$ admits a Borel $k$-coloring. Similarly, $c$ is called \textit{Baire measurable} if it is a Baire measurable function, and the \textit{Baire measurable chromatic number} of $G$, denoted $\chi_{BM}(G)$, is the least $k$ such that $G$ admits a Baire measurable $k$-coloring. For a survey covering this exciting emerging field, see \cite{KM}.

For a Borel graph $G$, we of course have $\chi(G) \leq \chi_{BM}(G) \leq \chi_B(G)$, but it is natural to ask just how large $\chi_{BM}(G)$ and $\chi_B(G)$ can be compared to $\chi(G)$. There are many known examples \cite{KM} where $\chi(G) = 2$ while $\chi_{BM}(G)$ and $\chi_B(G)$ are infinite. However, for graphs of bounded degree $d$, Kechris, Solecki, and Todorcevic \cite{KST} proved $\chi_B(G) \leq d+1$. We therefore restrict our attention to bounded degree graphs for the remainder of the paper.

In \cite{M15}, Marks proved that this bound is sharp, even for acyclic $G$ (so in particular, $\chi(G)=2$). Thus, $\chi_B(G)$ can be arbitrarily large compared to $\chi(G)$. On the other hand, for Baire measurable chromatic numbers, Conley and Miller proved the following \cite{CM} (Theorem B):

\begin{theorem}\label{th:cm}
Let $G$ be a locally finite Borel graph such that $\chi(G) < \aleph_0$. Then $\chi_{BM}(G) \leq 2\chi(G) - 1$.
\end{theorem}

The question ``How close to this bound can we get?" still remains. Previously, not much seems to have been known regarding this: In fact, Kechris and Marks pose the following problem \cite{KM} (Problem 4.7):

\begin{prob}\label{prob:bm1}
Is there a bounded degree Borel graph $G$ for which $\chi_{BM}(G) > \chi(G)+1$?
\end{prob}

The graphs constructed by Marks in \cite{M15} are not hyperfinite (see Section \ref{sec:measure} for a definition). Furthermore, an analogue of Theorem \ref{th:cm} holds for measure chromatic numbers if the extra assumption of hyperfiniteness is added (see Theorem  \ref{th:cmmeasure}). This lead to the question of whether the $2\chi(G)-1$ bound held for Borel chromatic numbers in the hyperfinite setting (\cite{KM}, Question 5.19). In \cite{CJMST}, though, Marks' techniques were adapted to the hyperfinite setting, giving a negative answer to this question. 

In this paper, however, we note that a certain strengthening of the hyperfiniteness assumption \textit{is} enough to get this bound. Using techniques similar to those in \cite{CM} and some results from \cite{M09}, we prove in Section \ref{sec:2ends} the following analogue of Theorem \ref{th:cm}:

\begin{theorem}\label{th:2ends}
Let $G$ be a locally finite Borel graph such that $\chi(G) < \aleph_0$, and such that every connected component of $G$ has two ends. Then $\chi_B(G) \leq 2\chi(G) - 1$.
\end{theorem}

See section \ref{sec:2ends} for a definition of two endedness. Also note that this condition is indeed a strengthening of hyperfiniteness \cite{M09}.

Similarly little seems to have been known regarding the sharpness of this bound. In fact, one of the goals of the project which led to this paper was to resolve the following:

\begin{prob}\label{prob:2ends}
Is there a bounded degree Borel graph $G$ whose connected components all have two ends for which $\chi_{B}(G) > \chi(G)+1$?
\end{prob}

In this paper, we answer Problems \ref{prob:bm1} and \ref{prob:2ends} as strongly as possible, proving the bounds in Theorems \ref{th:cm} and \ref{th:2ends} are sharp:

\begin{theorem}\label{th:main}
Let $k \geq 3$. There is a Borel $3(k-1)^2$-regular graph, say, $G_k$, such that all the connected components of $G_k$ have two ends, $\chi(G_k)=k$, and $\chi_{BM}(G_k)=\chi_B(G_k)=2k-1$.
\end{theorem}

The graphs $G_k$ will be arise in the following way: A \textit{marked group} (in this paper) is a pair $(\Gamma,S)$, where $\Gamma$ is a (typically infinite) finitely generated group and $S$ is a finite symmetric set of generators for it not containing the identity. When there is no confusion, we will sometimes refer to a marked group by its underlying group. Consider the group action $\Gamma \curvearrowright 2^\Gamma$ given by 
\begin{equation}\label{eq:shift}
    (g \cdot x)(h) = x(g\inv h)
\end{equation}
for $g,h \in \Gamma$ and $x \in 2^\Gamma$. This is called the \textit{left shift action}. When $2^\Gamma$ is given the product topology, this action is clearly continuous. Let \begin{equation}
    F(2^\Gamma) = \{x \in 2^\Gamma \mid  \forall g \in \Gamma-\{\id\}, g \cdot x \neq x \ \}.
\end{equation}
This is a $G_\delta$ subspace of $2^\Gamma$, hence a Polish space. We can therefore form a Borel graph on $F(2^\Gamma)$ by putting an edge between $x$ and $y$ exactly when $s \cdot x = y$ for some $s \in S$. This is called the \textit{shift graph} of $(\Gamma,S)$. We will always refer to the shift graph by its underlying set, $F(2^\Gamma)$. The graphs $G_k$ will all have the form $F(2^{\Gamma_k})$ for some marked group $\Gamma_k$.

Let $\cay(\Gamma)$ be the \textit{Cayley graph} of $(\Gamma,S)$. This is the graph on $\Gamma$ given by putting an edge between group elements $g$ and $h$ exactly when $sg=h$ for some $s \in S$. Clearly, as a (discrete) graph, $F(2^\Gamma)$ is isomorphic to a disjoint union of continuum many (if $\Gamma$ is infinite) copies of $2^\Gamma$. It is therefore natural to expect to get some information on the descriptive combinatorics of $F(2^\Gamma)$ from the graph $\cay(\Gamma)$. However, in \cite{FW} (Theorem 1), Weilacher showed that $\cay(\Gamma)$ is not enough to determine $\chi_B(F(2^\Gamma))$ or $\chi_{BM}(F(2^\Gamma))$:
\begin{theorem}\label{th:fw}
Let $k \geq 3$. There are marked groups $\Gamma$ and $\Delta$ with isomorphic Cayley graphs for which $\chi_B(F(2^\Delta))=\chi_{BM}(F(2^\Delta)) = k$ but $\chi_B(F(2^\Gamma))=\chi_{BM}(F(2^\Gamma)) = k+1$.
\end{theorem}
This led to the natural question:
\begin{prob}\label{prob:difference}
Are there marked groups $\Gamma$ and $\Delta$ with isomorphic Cayley graphs for which $\chi_B(F(2^\Gamma)) - \chi_B(F(2^\Delta)) > 1$? What about for Baire measurable chromatic numbers?
\end{prob}
We answer this as well by producing for each $k$ a marked group $\Delta_k$ whose Cayley graph is isomorphic to that of $\Gamma_k$, but for which $\chi_B(F(2^{\Delta_k})) = \chi_{BM}(F(2^{\Delta_k})) = k+1$. Thus we get
\begin{corollary}\label{cor:difference}
Let $k$ be a natural number. There are marked groups $\Gamma$ and $\Delta$ with isomorphic Cayley graphs but for which $\chi_B(F(2^\Gamma)) - \chi_B(F(2^\Delta)) = \chi_{BM}(F(2^\Gamma)) - \chi_{BM}(F(2^\Delta)) = k$. 
\end{corollary}

In Section 3 we define the marked groups $\Delta_k$ and $\Gamma_k$ and compute their various chromatic numbers. In Section 4 we note that everything said in this paper about Baire measurable chromatic numbers can also be said about measure chromatic numbers in the hyperfinite setting.


\section{Graphs Whose Connected Components All Have Two Ends}\label{sec:2ends}


In this section we prove Theorem \ref{th:2ends}. The proof uses little more than some results of Miller from \cite{M09}, but nevertheless the result seems to be new, and may be of interest to some.

Let $G$ be a graph on a set $X$. If $A \subset X$, we denote by $G \res A$ the graph $G \cap (A \times A)$ on $A$. We call $G$ \textit{connected} if it has one connected component, and $A$ \textit{connected} if $G \res A$ is connected. 

A \textit{path} between vertices $x$ and $y$ is a finite sequence $x=x_0,\ldots,x_n=y$ such that $(x_i,x_{i+1}) \in G$ for all $i$. In this situation $n$ is called the \textit{length} of the path. Note that a graph is connected if and only if there is a path between any two of its vertices. The \textit{path distance} between $x$ and $y$ is the smallest $n$ such that there is a path of length $n$ between $x$ and $y$, or $\infty$ if there is no path between $x$ and $y$. The \textit{path distance} between two sets of vertices $A$ and $B$ is the smallest path distance between any pair of vertices $x \in A$ and $y \in B$. A graph is called acyclic if it admits no paths as above with $x_0 = x_n$, but no other repeats among the $x_i$'s. 

Now assume $G$ is connected and locally finite. We say a subset $F \subset X$ \textit{divides $G$ into $n$ parts} if $G \res (X - F)$ has $n$ infinite connected components. We say $G$ \textit{has $n$ ends} if there is a finite set $F$ dividing $G$ into $n$ parts, but no such $F$ dividing $G$ into $m$ parts for any $m > n$. Note that if $G$ has $n$ ends, we can finite a finite set $F$ dividing it into $n$ parts such that $F$ is furthermore connected. It should be noted this definition is different in general from the one used in \cite{M09}, but is equivalent in the locally finite case.

Now let $G$ be a localy finite Borel graph on a space $X$ whose connected components all have two ends. Denote by $[G]^{<\infty}$ the standard Borel space of finite connected subsets of $X$. Let $\Phi \subset [G]^{<\infty}$ be the set of sets which divide their connected component into two parts. Miller (Lemma 5.3) proves that there is a maximal Borel set $\Psi' \subset \Phi$ whose members are pairwise disjoint. An easy modification of their proof shows that we can instead get a maximal Borel set $\Psi \subset \Phi$ such that the path distance between any two distinct members of $\Psi$ is at least 4. Fix such a $\Psi$.

Let $\T$ be the set of pairs $(S,T)$ with $S,T \in \Psi$ such that $S \neq T$ and there is a path from $S$ to $T$ which avoids all other points of $\bigcup \Psi$. Miller proves that $\T$ is an acyclic graph on $\Psi$, that $S$ and $T$ are connected in this graph if and only if they are subsets of the same connected component of $G$, and that every element of $\Psi$ is $\T$-adjacent to at most two other elements (Lemma 5.5). (Strictly speaking, they prove these things for $\Psi'$, but the proofs clearly still apply to $\Psi$.)

\begin{lemma}\label{lem:maximal}
Every $S \in \Psi$ is $\T$-adjacent to exactly two other elements.
\end{lemma}
\begin{proof}
Suppose some $S \in \Psi$ has fewer than two $\T$-neighbors. Let $C$ be the connected component of $S$. Let $C_-$ and $C_+$ be the two infinite connected components of $G \res (C - S)$. WLOG, $C_+$ must contain no sets in $\Psi$. This follows from the fact that any $T \in \Psi$ with $T \subset C$ must be $\T$-connected to $S$. 

Let $N$ be the set of points in $C_+$ whose path distance from $S$ is exactly 4. $N$ is finite since $G$ is locally finite. We claim $N$ divides $C$ into 2 parts: By K{\"o}nig's Lemma, we can find an injective sequence $\{x_n \mid n \in \omega\}$ of points in $C_+$ such that $(x_n,x_{n+1}) \in G$ for all $n$. Since $G$ is locally finite, there must be some $M$ for which for all $n \geq M$, the path distance between $x_n$ and $S$ is at least 5. Then the sequence $\{x_n \mid n \geq M\}$ does not pass through $N$, so it is contained in an infinite connected component of $G \res (C - N)$. Also, $C_{-}$ is contained in an infinite connected component of $G \res (C - N)$, so it suffices to show there is no path from $C_-$ to $x_M$ avoiding $N$. This is clear, though, as any path from $C_-$ to $x_M$ must pass through $S$, say at the point $y$, since $x_M \in C_+$. Then since the path distance from $S$ to $x_M$ is greater than 4, there must be some point in $C_+$ along our path from $y$ to $x_M$ whose path distance from $S$ is exactly 4.

Let $D$ be the infinite connected component of $G \res (C - N)$ not containing $S$. Let $N' \subset N$ be the set of elements of $N$ adjacent to a point in $D$. Then $N'$ still divides $C$ into 2 parts. Furthermore, we can find a finite subset $A \subset D$ such that $N' \cup A$ is connected. Then $N' \cup A \in \Phi$, and furthermore since every point in $D$ has path distance at least 5 from $S$, the path distance between $S$ and $N' \cup A$ is 4. However, since we assumed $C^+$ contains no sets in $\Psi$, this contradicts the maximality of $\Psi$.
\end{proof}

\begin{lemma}\label{lem:divided}
Every connected component of $G \res (X - \bigcup \Psi)$ is finite.
\end{lemma}

\begin{proof}
Let $x \in (X - \bigcup \Psi)$. Let $C$ be the connected component of $x$ in the graph $G$, and let $D$ be the connected component of $x$ in the graph $G \res (C - \bigcup \Psi)$. We want to show $D$ is finite.

By maximality there is some element of $\Psi$ contained in $C$. Then, by Lemma \ref{lem:maximal} along with the fact that $\T$ is acyclic, we can label the elements of $\Psi$ contained in $C$ as $\{S_n \mid n \in \Z\}$, where the indices are chosen such that $(S_n,S_m) \in \T$ if and only if $|n - m| = 1$. By definition of $\Phi$, for each $n$ the graph $G \res (C - S_n)$ has two infinite connected components, call them $C_{n,-}$ and $C_{n,+}$. By definition of $\T$, the sets $S_m$ for $m > n$ must all lie in the same connected component of $G \res (C- S_n)$, and likewise for the sets $S_m$ for $m < n$. Therefore, by relabelling if necessary, we can assume $S_m \subset C_{n,+}$ for all $m > n$ and $S_m \subset C_{n,-}$ for all $m < n$.

Now, suppose $D$ is infinite. Then, for each $n$, either $D \subset C_{n,+}$ or $D \subset C_{n,-}$. Consider integers $n$, points $y \in S_n$, and paths from $x$ to $y$. Choose $n,y$, and such a path such that this path is of minimal length among all such choices. Then this path cannot pass through any sets $S_m$ for $m \neq n$. WLOG assume $D \subset C_{n,+}$. We claim $D \subset C_{n+1,-}$. If not $D \subset C_{n+1,+}$, but then $D$ and $S_n$ are in different connected components of $G \res (C - S_{n+1})$, so there can be no path from $x$ to $S_n$ avoiding $S_{n+1}$, a contradiction. Therefore $D \subset C_{n,+} \cap C_{n+1,-}$, so this intersection is infinite. This implies, however, that the finite set $S_n \cup S_{n+1}$ divides $G \res C$ into at least three parts, a contradiction.
\end{proof}

We can now prove Theorem \ref{th:2ends}:

\begin{proof}

For each $S \in \Psi$, let $S^* = S \cup \{x \in X \mid \exists y \in S \ (x,y) \in G\}$. Since $G$ is locally finite and each $S$ is finite, each $S^*$ is finite. Let $B^* = \bigcup_{S \in \Psi} S^*$. $B^*$ is Borel since $\Psi$ is Borel. Since distinct $S$'s had path distances of at least 4 between them, distinct $S^*$'s have path distances of at least 2 between them. Thus, every connected component of $G \res B^*$ is a subset of some $S^*$. In particular these connected components are all finite. Therefore, by the Lusin-Novikov Uniformization Theorem (See \cite{K95}, Lemma 18.12), there is a Borel $\chi(G)$-coloring, say $c_1^*:B^* \rightarrow \{1,2,\ldots,\chi(G)\}$ of $G \res B^*$. Let $B = B^* - c_1^{*-1 }(\{\chi(G)\})$ and $c_1 = c_1^* \res B$. Then $B$ is Borel and $c_1$ is a Borel $(\chi(G)-1)$-coloring of $G \res B$.

We claim that the connected components of $G \res (X - B)$ are also all finite. Suppose to the contrary that $D \subset (X-B)$ is some infinite connected component. Let $C$ be the connected component of $G$ containing $D$. We first claim that $D$ must contain infinitely many points not in $B^*$. If not, then $D$ contains infinitely points from $(B^* - B)$, and only finitely many not in $B^*$. By construction, though, $B^* - B$ is independent, so since $D$ is connected, for every $y \in (B^*-B) \cap D$, there must be some $x \in D - B^*$ with $(x,y) \in G$. Thus there is some $x \in D - B^*$ connected to infinitely many such $y$'s, contradicting local finiteness. Therefore, by Lemma \ref{lem:divided}, there are $x,y \in D - B^*$ such that $x$ and $y$ are in different connected components of $G \res (C - \bigcup \Psi)$. Let $x=x_0,x_1,\ldots,x_n=y$ be a path from $x$ to $y$ consisting of points in $D$. Then there must be some $S \in \Psi$ and some $0 < i < n$ such that $x_i \in S$. Then $x_{i-1},x_{i}$, and $x_{i+1}$ are all in $S^*$. Since there are some edges between them, they can't all be assigned the color $\chi(G)$ by $c_1^*$, but this means at least one of them is in $B$, a contradiction.

Therefore, again by the Lusin-Novikov Uniformization Theorem, there is a Borel $\chi(G)$, coloring, say, $c_2:(X-B) \rightarrow \{\chi(G),\ldots,2\chi(G)-1\}$, of $G \res (X - B)$. Since $c_1$ and $c_2$ use disjoint sets of colors, $c_1 \cup c_2$ is a Borel $(2\chi(G)-1)$-coloring of $G$.
\end{proof}


\section{The Construction}\label{sec:main}


Fix $k \geq 3$. In this section, we define the marked groups $\Gamma_k$ and $\Delta_k$ promised in Section \ref{sec:intro}.

We start with a finite marked group: Let $Z_k$ denote the cyclic group of order $k$, which we will identify with the integers modulo $k$. Consider the group $Z_k \times Z_k$ with generating set $S = \{(a,b) \mid 0 < a,b < n \}$. Let $H$ be the Cayley graph of this finite marked group. We'll think of the vertices of $H$ as sitting on a $k$ by $k$ grid, with the horizontal axis corresponding to the first coordinate and the vertical to the second. Accordingly, by a $\textit{row}$ of $H$ we mean a set of the form $\{(a,b) \mid a \in Z_k\}$ for some fixed $b \in Z_k$, and by a $\textit{column}$ of $H$ we mean a set of the form $\{(a,b) \mid b \in Z_k\}$ for some fixed $a \in Z_k$.

\begin{figure}
\centering
\includegraphics[width=0.5\textwidth]{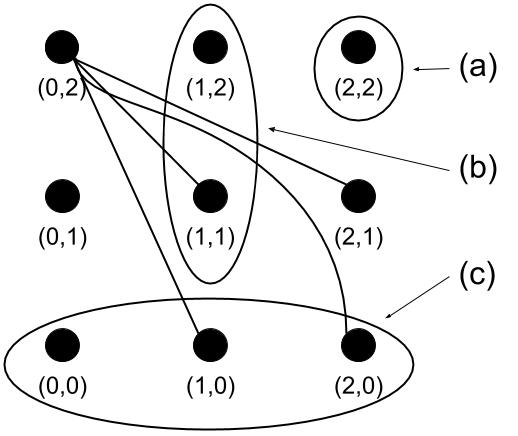}
\caption{\label{fig:hvz} A drawing of the graph $H$ for $k=3$. The edges shown are exactly those meeting $(0,2)$. The three types of independent sets are shown in circles and labeled: cardinality one (a), vertical (b), and horizontal (c). }
\end{figure}

An \textit{independent} subset of a graph is a pairwise-non-adjacent set of vertices. Thus, a coloring is just a partition of the set of vertices into independent sets. Note that any independent subset of $H$ of size greater than one must be either completely contained in some row, or completely contained in some column, (and not both). Call such sets \textit{horizontal} and \textit{vertical}, respectively (See Figure \ref{fig:hvz}).

\begin{lemma}\label{lem:orientation}
Let $c:Z_k \times Z_k \rightarrow \{1,2,\ldots,2k-2\}$ be a $(2k-2)$-coloring of $H$. Exactly one of the following holds:
\begin{itemize}
    \item Every row contains a horizontal color set.
    \item Every column contains a vertical color set.
\end{itemize}
\end{lemma}
\begin{proof}
Since there are $k$ rows and $k$ columns, for both to hold simultaneously would require $2k$ colors. Therefore at most one holds.

Suppose neither holds. Then there is some column $C$ and some row $R$ such that $C$ does not contain a horizontal color set and $R$ does not contain a vertical color set. Then every point in $R \cup C$ must have a different color, but $|R \cup C| = 2k-1$. Therefore at least one holds.
\end{proof}

We call $c$ as in the lemma a \textit{horizontal} coloring if the first condition holds, and a \textit{vertical} coloring if the second holds.

We can now define the marked group $\Delta_k$: It will be the group $(Z_k \times Z_k) \times \Z$, with generating set $S \times \{-1,0,1\}$. Let $G$ be the Cayley graph of $\Delta_k$. It's easy to see $\chi(G) = k$: A $k$-coloring is given by sending the element $((a,b),n)$ to $a$ for all $n \in \Z$ and $0 \leq a,b < k$. Also note that $G$ has two ends, as desired.

For each $n \in \Z$, the restriction of $G$ to the $(Z_k \times Z_k)$-orbit $(Z_k \times Z_k) \times \{n\}$ can be identified with $H$ in the obvious way. Thus, if $c:(Z_k \times Z_k) \times \Z \rightarrow \{1,2,\ldots,2k-2\}$ is a $(2k-2)$-coloring of $G$, the restriction of $c$ to the orbit $(Z_k \times Z_k) \times \{n\}$ is, for each $n$, either a horizontal coloring or a vertical coloring. In the $k$-coloring defined in the previous paragraph, all these restrictions were horizontal. The next lemma says that this was no accident:

\begin{figure}
\centering
\includegraphics[width=1\textwidth]{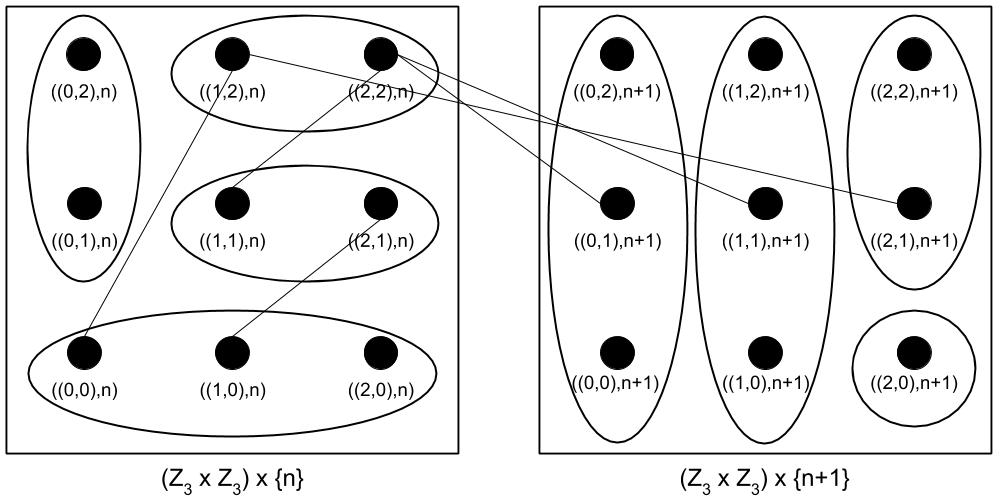}
\caption{\label{fig:invariance} A visual explanation of the proof of Lemma \ref{lem:invariance} in the case $k=3$. The two squares enclose neighboring $(Z_k \times Z_k)$-orbits. The cirlces represent color sets within each orbit. Most edges are omitted, but some are included to show any horizontal color set from the first orbit must admit an edge to every vertical color set from the second orbit. Others are included to show that horizontal color sets in different rows of a single orbit always have edges between them. The same is true for vertical color sets in different columns.}
\end{figure}

\begin{lemma}\label{lem:invariance}
Let $c:(Z_k \times Z_k) \times \Z \rightarrow \{1,2,\ldots,2k-2\}$ be a $(2k-2)$-coloring of $G$. Exactly one of the following holds:
\begin{itemize}
    \item The restriction of $c$ to every $(Z_k \times Z_k)$-orbit is horizontal.
    \item The restriction of $c$ to every $(Z_k \times Z_k)$-orbit is vertical.
\end{itemize}
\end{lemma}
\begin{proof}
By symmetry, it suffices to show that if the restriction of $c$ to $(Z_k \times Z_k) \times \{n\}$ is horizontal, then so is the restriction to $(Z_k \times Z_k) \times \{n+1\}$. Suppose instead that it is vertical. For $1 \leq i \leq k$, let $R_i$ be a horizontal color set contained in the $i$-th row of $(Z_k \times Z_k) \times \{n\}$, and let $C_i$ be a vertical color set contained in the $i$-th column of $(Z_k \times Z_k) \times \{n+1\}$. Observe that for every $1 \leq i,j \leq k$, there is at least one edge between $R_i$ and $C_j$ (See Figure \ref{fig:invariance}). Furthermore, if $i \neq j$, there is at least one edge between $R_i$ and $R_j$, as well as between $C_i$ and $C_j$ (Again see Figure \ref{fig:invariance}). Therefore each $R_i$ and $C_j$ must have a distinct color, but this requires $2k$ colors.
\end{proof}

This leads us to a natural definition of the marked group $\Gamma_k$: Let $\varphi \in \textnormal{Aut}(Z_k \times Z_k)$ be the coordinate swapping map: $\varphi(a,b) = (b,a)$. $\Gamma_k$ will be the semi-direct product $(Z_k \times Z_k) \rtimes_{1 \mapsto \varphi} \Z$, again with generating set $S \times \{-1,0,1\}$. Observe that the following gives an isomorphism between the Cayley graphs of $\Gamma_k$ and $\Delta_k$:
\begin{equation}\label{eq:iso}
    ((a,b),n) \mapsto ((a,b),n) \textnormal{ for } n \textnormal{ even,     } ((a,b),n) \mapsto ((b,a),n) \textnormal{ for } n \textnormal{ odd},
\end{equation}
where $a,b \in Z_k$ and $n \in \Z$. Thus we still have $\chi(F(2^{\Gamma_k})) = \chi(\cay(\Gamma_k)) = k$, and this Cayley graph still  has two ends as desired. We now compute the Borel and Baire measurable chromatic numbers of $F(2^{\Gamma_k})$, proving Theorem \ref{th:main}:

\begin{prop}
$\chi_{B}(F(2^{\Gamma_k})) = \chi_{BM}(F(2^{\Gamma_k})) = 2k-1$.
\end{prop}
\begin{proof}
Theorem \ref{th:2ends} gives us the upper bound $\chi_B(F(2^{\Gamma_k})) \leq 2k-1$, so it remains to show there is no Baire measurable $(2k-2)$-coloring of $F(2^{\Gamma_k})$.

Suppose first that $c:(Z_k \times Z_k) \rtimes_{1 \mapsto \varphi} \Z \rightarrow \{1,2,\ldots,2k-2\}$ is a $(2k-2)$-coloring of $\cay(\Gamma_k)$. Note that the isomorphism (3) sends $(Z_k \times Z_k)$-orbits to $(Z_k \times Z_k)$-orbits, but preserves the notions of ``horizontal'' and ``vertical'' for those with even $\Z$-coordinate and flips those notions for those with odd $\Z$-coordinate. Thus Lemma \ref{lem:invariance} has the following consequence for $\Gamma_k$: If for some $n$ the restriction of $c$ to $(Z_k \times Z_k) \times \{n\}$ is horizontal, the restiction to $(Z_k \times Z_k) \times \{n+1\}$ must be vertical, and vice versa.

Now suppose $c:F(2^{\Gamma_k}) \rightarrow \{1,2,\ldots,2k-2\}$ is a Baire measurable $(2k-2)$-coloring. Define the map $d:F(2^{\Gamma_k}) \rightarrow \{1,2\}$ by sending a point $x$ to 1 if the restriction of $c$ to the $(Z_k \times Z_k)$-orbit of $x$ is horizontal, and 2 if it is vertical. It is clear that $d$ is Baire measurable since $c$ was. By the previous paragraph, $d(x) \neq d(((0,0),1) \cdot x)$ for all $x$.

Now consider $\Z$ with generators $\{\pm 1\}$, and let $g:F(2^\Z) \rightarrow F(2^{\Gamma_k})$ be the map given by 
\begin{equation}
    g(y)((a,b),n) = \begin{cases} y(n) \textnormal{ if } (a,b)=(0,0) \\ 0 \textnormal{ else.}  \end{cases}
\end{equation}
Then $g$ is continuous, and $g(1 \cdot y) = ((0,0),1) \cdot g(y)$ for all $y$. Therefore $d \circ g$ is a Baire measurable 2-coloring of the shift graph $F(2^\Z)$. It was established in $\cite{KST}$, though, that $\chi_{BM}(F(2^\Z)) = 3$.
\end{proof}

Finally, we compute the Borel and Baire measurable chromatic numbers of $F(2^{\Delta_k})$, which gives Corollary \ref{cor:difference} as promised:

\begin{prop}\label{prop:delta}
$\chi_B(F(2^{\Delta_k})) = \chi_{BM}(F(2^{\Delta_k})) = k+1$.
\end{prop}
\begin{proof}
We first show there is no Baire measurable $k$-coloring $c:F(2^{\Delta_k}) \rightarrow \{1,2,\ldots,k\}$. Suppose we had such a coloring. Observe that all $k$-colorings of the Cayley graph of $\Delta_k$ look essentially like the one defined before Lemma \ref{lem:invariance}: Up to a relabeling of the colors, they either assign the color $a$ to $((a,b),n)$ for all $b$ and $n$, or the color $b$ to $((a,b),n)$ for all $a$ and $n$. In particular, the elements $g$ and $((0,0),1) \cdot g$ always have the same color.

Therefore, if we let $C_i = c\inv(\{i\})$ for each $i$, each $C_i$ is sent to itself by the action of the element $((0,0),1)$. Since the order of this element is infinite, a standard argument (see Theorem 8.46 in \cite{K95}) shows each $C_i$ is either meager or comeager. Since the $C_i$'s partition $F(2^{\Delta_k})$, at least one, say, $C_{i_0}$, must be meager. The sets $((a,b),0) \cdot C_{i_0}$ for $a,b \in Z_k$ cover $F(2^{\Delta_k})$, though, so this is a contradiction.

It remains to construct a Borel $(k+1)$-coloring $c:F(2^{\Delta_k}) \rightarrow \{1,2,\ldots,k+1\}$. 

A subset of $x$ is called \textit{$r$-discrete}, for $r$ a natural number, if the path distance between any two points in $A$ is greater than $r$. It is an easy corollary of Proposition 4.2 in \cite{KST} that if $G$ is a Borel graph of bounded degree, then $X$ contains a Borel maximal $r$-discrete subset for every $r$.

Applying this, let $A \subset F(2^{\Delta_k})$ be a Borel maximal $3k$-discrete set. Then every $(Z_k \times Z_k)$-orbit contains at most one element of $A$. For every $x \in A$, color the $(Z_k \times Z_k)$-orbit of $x$ by setting $c(((a,b),0) \cdot x) = a$ for $1 \leq a,b \leq k$.

We now color the $(Z_k \times Z_k)$-orbits between those meeting $A$. Let $x \in A$ with $(Z_k \times Z_k)$-orbit $E$, and let $N$ be the smallest positive number such that $((0,0),N) \cdot E$ contains a point of $A$. Call that point $y$. Also note $N > 3k$ by definition of $A$. There are elements $1 \leq a_0,b_0 \leq k$ such that $y = ((a_0,b_0),N) \cdot x$. Also let $E_n$ denote the orbit $((0,0),n) \cdot E$ for $n \in \Z$, so for example $y \in A \cap E_N$. We need to extend $c$ by coloring all the $E_n$'s for $0 < n < N$. We'll proceed one $n$ at a time:

Given a $(Z_k \times Z_k)$-orbit $E'$ colored already by $c$ and a positive integer $n$, let $c_n(E')$ denote the coloring on $((0,0),n) \cdot E'$ given by $c_n(E')(z) = c( ((0,0),-n) \cdot z)$. We could try to extend our coloring $c$, by coloring $E_1$ with $c_1(E)$, then $E_2$ with $c_1(E_1)$, and so on. If we happened to have $c \res E_N = c_N(E)$, this would work out, but otherwise we will have a conflict. We can use our additional color to fix this:

Now, color $E_1$ by using $c_1(E)$, but then swapping the color $k$ with the color $k+1$. Since $c \res E$ does not use the color $k+1$, this is OK. Then $c \res E_1$ does not use the color $k$, so we can color $E_2$ by using $c_1(E_1)$, but then swapping the color $a_0$ with the color $k$. Then $c \res E_2$ does not use the color $a_0$, so we can color $E_3$ by using $c_1(E_2)$, but then swapping the color $k+1$ with the color $a_0$. Now $c \res E_3$ does not use the color $k+1$, and furthermore it looks like $c_3(E)$, but with the colors $k$ and $a_0$ swapped. Note that by performing this swap, we have arragned that $c_{N-3}(E_3)$ agrees with $c \res E_N$ on the $a_0$-th row.

We can repeat this process $k$ times, so that for each $i \leq k$, $E_{3i}$ will not use the color $k+1$ and $c_{N-3i}(E_{3i})$ will agree with $c \res E_N$ on $i$ rows. In particular, we will have $c_{N-3k}(E_{3k}) = c \res E_N$. Thus, we can color the remaining orbits $E_{3k+i}$ for $0 < i < N-3k$ using $c_{i}(E_{3k})$. Thus we have a $(k+1)$-coloring $c$ as desired. Since $A$ was Borel, it is clear that $c$ is Borel, so we are done.

\end{proof}


\section{Measure Chromatic Numbers}\label{sec:measure}


In this section, we extend our results to the measurable setting.

Let $G$ be a Borel graph on a space $X$, now equipped with a Borel probability measure $\mu$. Just as we defined Borel and Baire measurable colorings, we can define $\mu$-\textit{measurable} colorings and the $\mu$-\textit{measurable chromatic number}, denoted $\chi_\mu(G)$. The \textit{measure chromatic number} of $G$, denoted $\chi_M(G)$, is the supremum of $\chi_\mu(G)$ over all Borel probability measures $\mu$ on $X$. 

An equivalence relation $E$ on $X$ is called \textit{Borel} if it is Borel as a subset of $X \times X$. $E$ is called \textit{finite} if its equivalence classes are all finite. $E$ is called \textit{hyperfinite} if it can be written as $E = \bigcup_{n \in \omega} E_n$ for some increasing sequence $E_n$ of finite Borel equivalence relations. $G$ is called \textit{hyperfinite} if its connected component equivalence relation is hyperfinite.

In \cite{CM} (Theorem A), Conley and Miller prove an analogue of Theorem \ref{th:cm} for measure chromatic numbers with the added assumption of hyperfiniteness:

\begin{theorem}\label{th:cmmeasure}
Let $G$ be a hyperfinite locally finite Borel graph such that $\chi(G) < \aleph_0$. Then $\chi_M(G) \leq 2\chi(G) - 1$.
\end{theorem}

As in the Baire measurable situation, the sharpness of this bound was previously unknown. All of the arguments we made in Section \ref{sec:main} in the Baire measurable setting still work in the measurable setting. Most crucially, we have $\chi_M(F(2^{\Z})) = 3$ just as we did for the Baire measurable chromatic number, and the arguments regarding $\chi_{BM}(F(2^{\Delta_k}))$ in the proof of Proposition \ref{prop:delta} still go through in the measure theoretic setting upon replacing ``meager'' and ``comeager'' with ``measure 0'' and ``measure 1'' respectively. Therefore,

\begin{prop}\label{prop:measure}
For all $k \geq 3$, $\chi_M(F(2^{\Gamma_k})) = 2k-1$ and $\chi_M(F(2^{\Delta_k})) = k+1$.
\end{prop}

As was noted in the introduction, these graphs are hyperfinite since their connected components all have two ends. Thus the bound in Theorem \ref{th:cmmeasure} is indeed sharp:

\begin{theorem}
Let $k \geq 3$. There is a Borel hyperfinite $3(k-1)^2$-regular graph, say, $G_k$, for which $\chi(G_k) = k$ but $\chi_M(G_k) = 2k-1$.
\end{theorem}

Similarly, alongside Theorem \ref{th:fw}, Weilacher \cite{FW} proves that there are marked groups with isomorphic Cayley graphs can have measure chromatic numbers which differ by one, but notes that it is open whether or not these numbers can differ by more than one. By Proposition \ref{prop:measure}, we have resolved this as well:

\begin{corollary}\label{cor:differencemeasure}
Let $k$ be a natural number. There are marked groups $\Gamma$ and $\Delta$ with isomorphic Cayley graphs but for which $|\chi_M(F(2^\Gamma)) - \chi_M(F(2^\Delta))| > k$. 
\end{corollary}






\section*{Acknowledgements}

We thank C. Conley for fruitful discussions and helpful comments on earlier drafts of this paper.

This work was partially supported by the ARCS foundation, Pittsburgh chapter. 

\pagebreak

\end{document}